\documentclass[11pt,a4paper]{article}

\usepackage{amssymb}
\usepackage{amsmath}
\usepackage{amsthm}
\usepackage{url}

\newtheorem{thm}{Theorem}[section]

\newtheorem{lem}[thm]{Lemma}

\newtheorem{fac}[thm]{Fact}

\def\N{{\mathbb N}}
\def\M{\mathrm{M}}

\def\H{\mathcal{H}}
\def\L{\mathcal{L}}
\def\R{\mathcal{R}}
\def\D{\mathcal{D}}
\def\T{\mathcal{T}}
\def\PT{\mathcal{PT}}

\def\rr{\mathbf{r}}
\def\ww{\mathbf{w}}
\def\LL{\mathbf{\Lambda}}
\def\II{\mathbf{I}}
\def\es{\varnothing}

\def\st{^\ast}
\def\ol#1{\overline{#1}}

\def\ig#1{\mathsf{IG}(#1)}
\def\pre{\,|\,}

\renewcommand\leq{\leqslant}
\renewcommand\geq{\geqslant}

\DeclareMathOperator\im{im} \DeclareMathOperator\kr{ker} \DeclareMathOperator\dom{dom} \DeclareMathOperator\fix{fix}

\begin{document}

\title{A note on free idempotent generated semigroups \\ over the full monoid of partial transformations}

\author{\textbf{Igor Dolinka}\\
\small Department of Mathematics and Informatics, University of Novi Sad,\\[-.33mm]
\small Trg Dositeja Obradovi\'ca 4, 21000 Novi Sad, Serbia\\[-.33mm]
\small E-mail: dockie@dmi.uns.ac.rs}

\date{}

\maketitle

\begin{abstract}
\noindent Recently, Gray and Ru\v skuc proved that if $e$ is a rank $k$ idempotent transformation of the set
$\{1,\dots,n\}$ to itself and $k\leq n-2$, then the maximal subgroup of the free idempotent generated semigroup over
the full transformation monoid $\T_n$ containing $e$ is isomorphic to the symmetric group $\mathcal{S}_k$. We prove
that the same holds when $\T_n$ is replaced by $\PT_n$, the full monoid of partial transformations on $\{1,\dots,n\}$.
\medskip

\medskip

\end{abstract}

\section{Introduction}

For a semigroup $S$, let $E=E(S)$ be the set of its idempotents. Then $E$ can be endowed with a structure of a partial
algebra by considering the products $ef,fe$ inherited from $S$, defined for $e,f\in E$ such that
$\{ef,fe\}\cap\{e,f\}\neq\es$. (Note that $ef\in\{e,f\}$ implies that $fe$ is an idempotent, and the same is true if
one switches the roles of $e,f$ in the latter statement.) Such products are called \emph{basic products}, and the
corresponding (unordered) pairs of idempotents are \emph{basic pairs}.

The \emph{free idempotent generated semigroup over $E$} (one can also say `over $S$' when $E=E(S)$) is defined by the
following presentation:
$$\ig{E} = \langle E\pre e\cdot f=ef\text{ such that }\{e,f\}\text{ is a basic pair}\,\rangle .$$
Here $ef$ denotes the product of $e$ and $f$ in $S$ (which is again an idempotent of $S$ and thus an element of $E$),
while $\cdot$ stands for the concatenation operation in the free semigroup $E^+$ (also to be interpreted as the
multiplication in its quotient $\ig{E}$). The term `free' is justified by the fact (see property (E1) below) that
$\ig{E}$ is the universal object in the category of all idempotent generated semigroups whose partial algebras of
idempotents are isomorphic to $E$.

It was conjectured in \cite{Nam} that any maximal subgroup of any semigroup of the form $\ig{E}$ is a free group; this
was exemplified by some early results, such as \cite{McE,NP,P1,P2}. However, Brittenham, Margolis and Meakin
\cite{BMM1} recently came up with a counterexample: they constructed a certain 72-element semigroup yielding a free
idempotent generated semigroup with a maximal subgroup isomorphic to $\mathbb{Z}\oplus\mathbb{Z}$, the rank 2 free
abelian group (a simpler example with the same maximal subgroup was given later by the author in \cite{D}, based on a
20-element regular band). In fact, as shown by Gray and Ru\v skuc in \cite{GR1}, \emph{every} group is isomorphic to a
maximal subgroup of $\ig{E(S)}$ for a suitably chosen $S$; if the `target' group is finitely presented, even a finite
semigroup $S$ suffices to achieve this goal.

Therefore, it appears as a quite natural problem to study maximal subgroups of free idempotent generated semigroups
$\ig{E(S)}$ for semigroups $S$ that `occur in nature', i.e.\ for the most classical, textbook examples of semigroups.
In \cite{BMM2}, Brittenham, Margolis and Meakin proved that if $Q$ is a division ring and $S=\M_n(Q)$, the full monoid
of $n\times n$ matrices ($n\geq 3$) over $Q$, then the maximal subgroup of $\ig{E(S)}$ whose identity element is a rank
1 idempotent matrix is isomorphic to $Q^\ast$, the multiplicative group of $Q$. In addition, it was explained in
\cite{BMM2} that the study of maximal subgroups of semigroups of the form $\ig{E}$ has much to do with algebraic
topology: these subgroups are in fact fundamental groups of connected components of a certain complex (called the
\emph{Graham-Houghton complex}) associated to $E$. As it turns out, the maximal subgroup of $\ig{E(S)}$ corresponding
to a $\D$-class $D$ of the subsemigroup $S'=\langle E(S)\rangle$ of $S$ ($\D$-classes of $S'$ and $\ig{E(S)}$ are in a
bijective correspondence) is isomorphic to the (up to isomorphism unique) maximal subgroup of $S$ contained in $D$ if
and only if the universal connected cover of the component of the Graham-Houghton complex corresponding to $D$ is
simply connected. This led the authors of \cite{BMM2} to conjecture that the maximal subgroup of the free idempotent
generated semigroup over $\M_n(Q)$, $n\geq 3$, corresponding to the $\D$-class of rank $k$ matrices, is isomorphic to
the general linear group $\mathrm{GL}_k(Q)$ if $k\leq n/2$ (and perhaps even if $k\leq n-2$). Quite recently, Gray and
Ru\v skuc \cite{GR2} proved the analogue of this conjecture for $\T_n$, the full transformation monoid on an
$n$-element set.

\begin{fac}[Main Theorem of \cite{GR2}]\label{GR}
Let $e$ be an idempotent transformation of the set $\{1,\dots,n\}$ such that $|\im(e)|=k\leq n-2$. Then the maximal
subgroup of $\ig{E(\T_n)}$ containing $e$ is isomorphic to the symmetric group $\mathcal{S}_k$.
\end{fac}

In this short note we show that the above result still holds true when $\T_n$ is replaced by $\PT_n$, the monoid of all
partial transformations of an $n$-element set, while $e$ is an idempotent partial transformation of $\{1,\dots,n\}$
whose rank (image size) is $k\leq n-2$. In the next section we review the necessary facts about the structure of
$\D$-classes of $\T_n$ and $\PT_n$, and recall a general method for obtaining presentations of maximal subgroups of
free idempotent generated semigroups. The results themselves are presented in the final section. Throughout the note we
assume familiarity with basic semigroup theory (in particular with Green's relations) that can be found e.g.\ in the
first sections of \cite{How}.

\section{Preliminaries}

\subsection{Green's Relations in $\T_n$ and $\PT_n$}

Throughout the note we compose (partial) functions from left to right and write them to the right of their arguments
(in accordance with \cite{GR2}), so that if $\alpha,\beta$ are partial transformations of the set $\N_n=\{1,\dots,n\}$
then $x(\alpha\beta)=(x\alpha)\beta$ for any $x\in\N_n$ for which both sides of the latter equality are defined. The
domain of the partial transformation $\alpha$ is denoted by $\dom(\alpha)$ and its image is $\im(\alpha)$. Further,
$\alpha$ defines a partition of $\dom(\alpha)\subseteq\N_n$, that is, a symmetric and transitive relation $\kr(\alpha)$
on $\N_n$ defined by $(x,y)\in\kr(\alpha)$ if and only if $x,y\in\dom(\alpha)$ and $\alpha(x)=\alpha(y)$. The size of
the image of $\alpha$, $|\im(\alpha)|$, is often referred to as the \emph{rank} of $\alpha$.

Detailed information on the structure of monoids $\T_n$ and $\PT_n$ is provided e.g.\ in the monograph of Ganyushkin
and Mazorchuk \cite{GM}. In particular, for any (partial) transformations $\alpha,\beta$ of $\N_n$ the following hold
both in $\T_n$ and, \emph{mutatis mutandis}, in $\PT_n$:
\begin{itemize}
\item $\alpha\,\D\,\beta$ if and only if $\alpha$ and $\beta$ have the same rank. Therefore, the $\D$-classes form a
chain of length $n$ in $\T_n$ and of length $n+1$ in $\PT_n$ (the extra $\D$-class in the latter is the lowest one,
containing the empty partial map). In both cases, the permutations of $\N_n$ form the top class, the group of units of
both the considered monoids.
\item $\alpha\,\L\,\beta$ if and only if $\im(\alpha)=\im(\beta)$ (since the composition of functions introduced here
is dual to the one defined in \cite{GM}).
\item $\alpha\,\R\,\beta$ if and only if $\kr(\alpha)=\kr(\beta)$, which includes the condition
$\dom(\alpha)=\dom(\beta)$.
\item $\alpha$ is an idempotent if and only $\im(\alpha)$ coincides with $\fix(\alpha)$, the set of fixed points of
$\alpha$. Consequently, $H_\alpha$, the $\H$-class of $\alpha$, is a group if and only if $\im(\alpha)$ is a
transversal of $\kr(\alpha)$. In that case, the identity element of $H_\alpha$ is the idempotent partial function
mapping each element from a given class of $\kr(\alpha)$ to the unique element of $\im(\alpha)$ lying in that class.
\end{itemize}

In the sequel, let $D_k$ denote the $\D$-class of $\PT_n$ comprising all rank $k$ partial transformations ($0\leq k\leq
n$), and let $D'_k=D_k\cap\T_n$ be the corresponding $\D$-class of $\T_n$ if $k\geq 1$. The $\R$-classes contained in
$D_k$ can be therefore indexed by partitions of subsets of $\N_n$ of cardinality $\geq k$ into $k$ classes (whence such
partitions of the whole $\N_n$ correspond to $\R$-classes contained in $D'_k$), while the $\L$-classes of $D_k$ (and of
$D'_k$) are indexed by subsets of $\N_n$ of cardinality $k$.  For example, a nice visual representation of $\D$-classes
of $\PT_3$ is given on p.60 of \cite{GM}.

\subsection{A Presentation for a Maximal Subgroup of $\ig{E}$}\label{pres}

Let $S$ be a monoid, $E=E(S)$, and $e\in E$; furthermore, let $D$ be the $\D$-class of $e$. In this subsection we
recall a procedure to obtain a presentation for $G_e$, the maximal subgroup of $\ig{E}$ containing $e$. For the case
when $S$ is regular this has been already known from \cite{Nam}, while the general case was deduced in \cite{GR1} from
the rewriting technique given by Ru\v skuc in \cite{R-JA}, which was, in turn, inspired by the classical
Reidemeister-Schreier method from combinatorial group theory. Bearing in mind the semigroups to which we apply these
general results, it is convenient to assume that $S$ is \emph{completely semisimple}, which means that $\D=\mathcal{J}$
and if $K$ is the ideal of $S$ generated by any member of $D$ and $K'=K\setminus D$, then the principal factor $K/K'$
is completely ($0$-)simple, that is, isomorphic to a Rees matrix semigroup $\mathcal{M}(I,H,\Lambda,P)$ with sandwich
matrix $P=(p_{\lambda i})_{\Lambda\times I}$ over $H\cup\{0\}$. As is well-known, there is no loss of generality in
assuming that there is a special index $1\in I\cap\Lambda$, whereas the $\H$-class $H_{11}$ is a group ($\cong H$)
containing $e$ as its identity element. As usual, $H_{i\lambda}=R_i\cap L_\lambda$ denotes the $\H$-class situated in
row $i$ and column $\lambda$. If this class is a group, then $e_{i\lambda}$ denotes the unique idempotent of $S$ it
contains.

For a word $\ww\in E\st$, let $\ol\ww$ denote the image of $\ww$ under the canonical homomorphism of $E\st$ into $S$:
in other words, $\ol\ww$ is just the element of $S$ obtained by multiplying in $S$ the idempotents the concatenation of
which is $\ww$. We say that a system of words $\rr_\lambda,\rr'_\lambda\in E\st$, $\lambda\in\Lambda$, is a
\emph{Schreier system of representatives} for $D$ if for each $\lambda\in\Lambda$:
\begin{itemize}
\item the right multiplications by $\ol{\rr_\lambda}$ and $\ol{\rr'_\lambda}$ are mutually inverse ($\R$-class
preserving) bijections $L_1\to L_\lambda$ and $L_\lambda\to L_1$, respectively;
\item each prefix of $\rr_\lambda$ coincides with $\rr_\mu$ for some $\mu\in\Lambda$ (for example, the empty word is
just $\rr_1$).
\end{itemize}
It is well-known that such a Schreier system always exists. In the following, we assume that one particular Schreier
system has been fixed.

In addition, we will assume that a mapping $i\mapsto\lambda_i$ has been specified such that $H_{i\lambda_i}$ is a group
(that is, $p_{\lambda_ii}\neq 0$): such $\lambda_i$ must exist for each $i\in I$ by the definition of a Rees matrix
semigroup, and it will be called the \emph{anchor} for the $\R$-class $R_i$.

Finally, call a \emph{square} a quadruple of idempotents $(e,f,g,h)$ in $D$ such that
$$\begin{array}{ccc}
e & \R & f\\[1mm]
\L && \L\\[1mm]
g & \R & h.
\end{array}$$
Then there are $i,j\in I$ and $\lambda,\mu\in\Lambda$ such that $e\in H_{i\lambda}$, $f\in H_{i\mu}$, $g\in
H_{j\lambda}$ and $h\in H_{j\mu}$. For an idempotent $\varepsilon\in S$ (which instantly turns out to belong to a
$\D$-class $D'$ of $S$ such that $D'\geq D$) we say that it \emph{singularizes} the square $(e,f,g,h)$ if any of the
following two cases takes place:
\begin{itemize}
\item[(a)] $\varepsilon e=e$ and $\varepsilon g=g$, while $e=f\varepsilon$; or
\item[(b)] $e=\varepsilon g$, along with $e\varepsilon=e$ and $f\varepsilon=f$.
\end{itemize}
Note that case (a) implies $\varepsilon f=f$, $\varepsilon h=h$, $e\varepsilon=e$ and $g=g\varepsilon=h\varepsilon$,
while conditions $\varepsilon e=e$, $f=\varepsilon f=\varepsilon h$, $g\varepsilon=g$ and $h\varepsilon=h$ follow from
(b). The square $(e,f,g,h)$ is \emph{singular} if it is singularized by some idempotent of $S$.

Now $G_e$ is isomorphic to the group defined by the presentation $\langle\Gamma\pre\mathfrak{R}\rangle$, where
$$\Gamma=\{X_{i\lambda}:\ p_{\lambda i}\neq 0\}$$
(recall that the condition $p_{\lambda i}\neq 0$ is equivalent to $H_{i\lambda}$ being a group, i.e.\ to the existence
of $e_{i\lambda}$), while we distinguish three types of defining relations in $\mathfrak{R}$:
\begin{itemize}
\item[(1)] $X_{i\lambda_i}=1$ for all $i\in I$;
\item[(2)] $X_{i\lambda}=X_{i\mu}$ for all $i\in I$ and $\lambda,\mu\in\Lambda$ such that $\rr_\lambda\cdot
e_{i\mu}=\rr_\mu$;
\item[(3)] $X_{i\lambda}^{-1}X_{i\mu} = X_{j\lambda}^{-1}X_{j\mu}$ for all $i,j\in I$ and $\lambda,\mu\in\Lambda$ such
that $(e_{i\lambda},e_{i\mu},e_{j\lambda},e_{j\mu})$ is a singular square in $D$.
\end{itemize}
We refer the reader to \cite[Theorem 5]{GR1} and its extensive argumentation for more details.

\section{The Result}

As already announced, our objective here is to prove the following.

\begin{thm}\label{main}
Let $e$ be an idempotent partial transformation of $\N_n$ such that $|\im(e)|=k\leq n-2$. Then the maximal subgroup of
$\ig{E}$, where $E=E(\PT_n)$, is isomorphic to the symmetric group $\mathcal{S}_k$.
\end{thm}

We make use of the following facts about free idempotent generated semigroups $\ig{E(S)}$ (where $S$ is arbitrary),
which are also mentioned and utilized in \cite{GR1}:
\begin{itemize}
\item[(E1)] There exists a natural (surjective) homomorphism $\phi:\ig{E(S)}\to S'$, where $S'$ denotes the subsemigroup of
$S$ generated by $E(S)$, whose restriction to $E(S)$ is the identity mapping. The basic pairs in $\ig{E(S)}$ are
exactly the same as in $S$.
\item[(E2)] $\phi$ induces a bijection between $\D$-classes of $\ig{E(S)}$ and those of $S'$. Furthermore, $\phi$ maps
the $\R$-class ($\L$-class) of $e\in E(S)$ onto the corresponding class of $e\in S'$, thus inducing a bijection between
the set of $\R$-classes ($\L$-classes) in the $\D$-class of $e$ in $\ig{E(S)}$ and the corresponding set in $S'$.
\item[(E3)] The restriction of $\phi$ onto $G_e$, the maximal subgroup of $\ig{E(S)}$ containing $e$, is a
homomorphism onto the maximal subgroup of $S'$ containing $e$.
\end{itemize}
It is a classical and celebrated result of Howie \cite{Ho1} that if $S=\T_n$, then $S'$ is the subsemigroup consisting
of all singular transformations (transformations of rank $<n$) and the identity mapping; in other words,
$S'=(\T_n\setminus\mathcal{S}_n)\cup\{\iota_n\}$. The same is true for $\PT_n$ by a result of Evseev and Podran
\cite{EP}: if $S=\PT_n$, then $S'$ includes all partial transformations of $\N_n$ except the nonidentical permutations.

Turning back to $\D$-classes $D_k$ and $D'_k$ of $\PT_n$ (i.e.\ of $\PT'_n$, provided $k<n$) and $\T_n$, respectively,
consisting of all rank $k$ members of their respective monoids, we have already remarked that in both of them, their
$\L$-classes can be indexed by the set $\LL$ of all $k$-element subsets of $\N_n$, representing the image of a partial
transformation. On the other hand, the $\R$-classes of $D_k$ can be indexed by the set $\II$ consisting of all
partitions of a subset of $\N_n$ of cardinality $\geq k$ into $k$ classes. A typical index $i\in\II$ will be of the
form $i=(A,\rho)$ where $A\subseteq\N_n$, $|A|\geq k$, and $\rho$ is an equivalence of $A$. Accordingly, $i$ is an
index of an $\R$-class contained in $D'_k$ if and only if $|A|=n$ and we will let $\II'\subseteq\II$ denote the set of
all $i\in\II$ with the latter property. Visually, one can think of $D'_k$ as a `horizontal slice' in $D_k$
corresponding to rows indexed by $\II'$. So, for any idempotent $e\in D_k$ there is a \emph{total} idempotent
transformation $e_0$ such that $e\,\D\,e_0$: just take $e_0$ to be of the same rank as $e$. Because of this easy remark
and (E2) above, there is no loss of generality in assuming that the idempotent partial transformation $e$ mentioned in
the formulation of Theorem \ref{main} is actually a total idempotent transformation of rank $k\leq n-2$; it will be
convenient for us to do so.

The key observation which eventually allows to reduce the verification of Theorem \ref{main} to Fact \ref{GR} is
contained in the next lemma.

\begin{lem}\label{L1}
Let $\alpha,\beta\in D_k$ be idempotent partial transformations such that $\alpha\,\R\,\beta$, whose domains are of
cardinality $m<n$. Then there exist total transformations $\alpha',\beta'\in D'_k$ such that
$(\alpha,\beta,\alpha',\beta')$ is a singular square of $\PT_n$.
\end{lem}

\begin{proof}
By the given conditions we have $\dom(\alpha)=\dom(\beta)=A$, $|A|=m$, and $\kr(\alpha)=\kr(\beta)=\rho$. Let
$C=\N_n\setminus A$ and fix an element $a_0\in A$ in an arbitrary way. We extend the (properly partial) transformations
$\alpha$ and $\beta$ to total transformations $\alpha'$ and $\beta'$, respectively, by defining $c\alpha'=a_0\alpha$
and $c\beta=a_0\beta'$ for all $c\in C$ and claim that these have the required properties.

First of all, note that if $\rho'$ denotes the equivalence of $\N_n$ obtained as an extension of $\rho$ by collapsing
$a_0/\rho$ and $C$ into a single $\rho'$-class (while all other $\rho$-classes remain the same, i.e.\ they are
$\rho'$-classes as well), then $\ker(\alpha')=\ker(\beta')=\rho'$. Hence, $\alpha'\,\R\,\beta'$. In addition, it is
obvious, by the very definition of $\alpha'$ and $\beta'$, that $\im(\alpha')=\im(\alpha)$ and
$\im(\beta')=\im(\beta)$, so $\alpha'\,\L\,\alpha$ and $\beta'\,\L\,\beta$. (In fact, it is quite straightforward to
verify that $\{\alpha,\beta,\alpha',\beta'\}$ forms a $2\times 2$ rectangular band.)

It remains to find an idempotent in $\mathcal{PT}_n$ (of rank $\geq k$) which singularizes this square. To this end,
define $\varepsilon \in \mathcal{PT}_n$ by
$$
i \varepsilon =
\begin{cases}
x \alpha & \mbox{if $i = x\beta$ where $x \in A$,} \\[1mm]
i & \mbox{otherwise}.
\end{cases}
$$
Obviously, $\varepsilon$ is a well-defined transformation, as $x \beta = y\beta$ implies $(x,y)\in\kr(\beta) =
\kr(\alpha)$ and thus $x\alpha = y\alpha$. To see that $\varepsilon^2 = \varepsilon$ observe that for all $x \in A$
$$
(x \alpha) \varepsilon =
\begin{cases}
(y \alpha) \alpha & \mbox{if $x \alpha = y\beta$ where $y \in A$,} \\[1mm]
x \alpha  & \mbox{otherwise}.
\end{cases}
$$
But if $x \alpha = y\beta$ then since $\alpha\,\R\,\beta$ and $\alpha^2 = \alpha$ it follows that $(y \alpha) \alpha =
y \alpha =  y (\beta \alpha) = (y \beta) \alpha = x \alpha \alpha = x \alpha$, and hence $\varepsilon^2 = \varepsilon$.
Since $\beta$ is an idempotent it follows that $\im(\beta) \subseteq \dom(\beta)=A$. Now $\varepsilon \alpha' =
\alpha'$ since for all $i \in \N_n$
\begin{align*}
i \varepsilon \alpha' & =
\begin{cases}
(x \alpha) \alpha' & \mbox{if $i = x \beta$ where $x \in A$,} \\[1mm]
i \alpha' & \mbox{otherwise}.
\end{cases} & & \\[2mm]
& =
\begin{cases}
x \alpha & \mbox{if $i = x \beta$  where $x \in A$,} \\[1mm]
i \alpha' & \mbox{otherwise}
\end{cases}
& & \mbox{(since $\alpha\,\L\,\alpha'$)}
\\[2mm]
& =
\begin{cases}
x \beta \alpha & \mbox{if $i = x \beta$  where $x \in A$,} \\[1mm]
i \alpha' & \mbox{otherwise}
\end{cases}
& & \mbox{(since $\alpha\,\R\,\beta$)} \\[2mm]
& = \begin{cases}
i \alpha & \mbox{if $i = x \beta$  where $x \in A$,} \\[1mm]
i \alpha' & \mbox{otherwise}
\end{cases}
& & \\[2mm]
& = i \alpha' & & \mbox{(since $x \beta \in A$).}
\end{align*}
Further routine calculations show $\beta \varepsilon = \alpha \varepsilon = \alpha$ and $\varepsilon \alpha = \alpha$.
Hence, $(\alpha, \beta, \alpha', \beta')$ is singularized by $\varepsilon$ (see case (a)).
%
\end{proof}

The following is immediate.

\begin{lem}\label{L2}
Let $\alpha,\beta,\gamma,\delta$ be total transformations of rank $k$ forming a square $(\alpha,\beta,\gamma,\delta)$
which is singularized in $\T_n$. Then it is a singular square in $\PT_n$ as well.
\end{lem}

\begin{proof}
This is trivial, since $\T_n$ is a subsemigroup of $\PT_n$; thus if $\varepsilon$ is a (total) transformation on $\N_n$
that singularizes the considered square in $\T_n$, it also singularizes the square in $\PT_n$.
\end{proof}

\begin{lem}\label{L3}
Let $\rr_\lambda,\rr'_\lambda\in E\st$, $E=E(\T_n)$, $\lambda\in\LL$, be a Schreier system of representatives for
$D'_k$ in $\T_n$. Then this is also a Schreier system of representatives for $D_k$ in $\PT_n$.
\end{lem}

\begin{proof}
For any $\lambda\in\LL$, let $L_\lambda$ be the $\L$-class of $D_k$ indexed by $\lambda$; then
$L'_\lambda=L_\lambda\cap D'_k$ is the corresponding $\L$-class of $\T_n$ contained in $D'_k$. By the given condition,
the right multiplication mappings by $\ol{\rr_\lambda}$ and $\ol{\rr'_\lambda}$ induce mutually inverse bijections
$L'_1\to L'_\lambda$ and $L'_\lambda\to L'_1$, respectively. In particular, $e\ol{\rr_\lambda}\,\R\,e$ and
$e\ol{\rr_\lambda}\ol{\rr'_\lambda}=e$. By Green's Lemma (see e.g.\ \cite[Lemma 2.2.1]{How}) the right translations by
$\ol{\rr_\lambda}$ and $\ol{\rr'_\lambda}$ considered above extend to ($\R$-class preserving) mutually inverse
bijections $L_1\to L_\lambda$ and $L_\lambda\to L_1$, respectively. The condition that the language $\{\rr_\lambda:\
\lambda\in\LL\}$ is prefix-closed stands unaltered, thus the lemma follows.
\end{proof}

Finally, we notice a consequence of the property (E3) and the characterization of group $\H$-classes in $\PT_n$ (see
\cite{GM}).

\begin{lem}\label{L4}
Let $e\in E(\PT_n)$ be of rank $k<n$. Then there is a surjective homomorphism from $G_e$, the maximal subgroup of
$\ig{E(\PT_n)}$ containing $e$, onto the symmetric group $\mathcal{S}_k$.
\end{lem}

\begin{proof}
Since $k<n$, the maximal subgroup of $\PT'_n$ containing $e$ coincides with the corresponding maximal subgroup of
$\PT_n$. Since the latter is known to be isomorphic to $\mathcal{S}_k$, the rest follows by (E3).
\end{proof}

It remains to put the available pieces together.

\begin{proof}[Proof of Theorem \ref{main}]
As already argued, we may assume that $e$ is an idempotent total transformation of $\N_n$ of rank $k$ belonging to the
$\H$-class $H_{11}\subseteq D_k$ of $\PT_n$. Let $G_e$ be the maximal subgroup of $\ig{E(\PT_n)}$ containing $e$, and
let $G'_e$ be the corresponding subgroup of $\ig{E(\T_n}$. The general construction in Subsection \ref{pres} yields
presentations for these two groups: $G_e=\langle\Gamma\pre\mathfrak{R}\rangle$ and
$G'_e=\langle\Gamma'\pre\mathfrak{R}'\rangle$. Both sets of defining relations can be in a natural way partitioned as
$\mathfrak{R}=\mathfrak{R}_1\cup\mathfrak{R}_2\cup\mathfrak{R}_3$ and
$\mathfrak{R}'=\mathfrak{R}'_1\cup\mathfrak{R}'_2\cup\mathfrak{R}'_3$ according to their `type' (1)--(3).

First of all, the fact that $\T_n$ is a submonoid of $\PT_n$ yields that $\Gamma'\subseteq\Gamma$. Moreover, note that
in the process of choosing anchor coordinates $\lambda_i\in\LL$ for each $i\in\II$, we can do it in two steps: first we
can fix a choice of anchors $\lambda_i$ for $i\in\II'$ (which are denoted $A(P)$ and concretely specified in \cite{GR2}
for an arbitrary partition $P\in\II'$ of $\N_n$), use it for constructing the presentation
$\langle\Gamma'\pre\mathfrak{R}'\rangle$ and then extend this mapping $i\mapsto\lambda_i$ to the whole of $\II$ in
order to obtain the presentation $\langle\Gamma\pre\mathfrak{R}\rangle$. In this way, we can make sure that
$\mathfrak{R}'_1\subseteq\mathfrak{R}_1$. Furthermore, Lemma \ref{L3} implies that
$\mathfrak{R}'_2\subseteq\mathfrak{R}_2$, while Lemma \ref{L2} yields $\mathfrak{R}'_3\subseteq\mathfrak{R}_3$.
Therefore, $\mathfrak{R}'\subseteq\mathfrak{R}$.

We proceed by showing that one can apply Tietze transformations to the presentation
$\langle\Gamma\pre\mathfrak{R}\rangle$ to the effect of removing all the generators from $\Gamma\setminus\Gamma'$.
Namely, let $i\in\II\setminus\II'$ and $\lambda\in\LL$ be arbitrary such that $H_{i\lambda}$ contains an idempotent
partial transformation $\beta$ (which means that $X_{i\lambda}\in\Gamma$). By construction, $H_{i\lambda_i}$ also
contains an idempotent which we denote by $\alpha$. Now Lemma \ref{L1} supplies an index $j\in\II'$ and idempotent
total transformations $\alpha'\in H_{j\lambda_i}$, $\beta'\in H_{j\lambda}$ of $\N_n$ such that
$(\alpha,\beta,\alpha',\beta')$ is a singular square in $D_k$. (To be quite precise, we do not assume that $\alpha\neq
\beta$: if $\alpha=\beta$, then accordingly $\alpha'=\beta'$, and the whole subsequent argument holds, although in a
trivial manner.) Therefore,
$$X_{i\lambda_i}^{-1}X_{i\lambda} = X_{j\lambda_i}^{-1}X_{j\lambda}$$
is a relation from $\mathfrak{R}$ (or a trivial relation), while $X_{i\lambda_i}=1$ is a relation of type (1) belonging
to $\mathfrak{R}$. Combining the two yields
$$X_{i\lambda} = X_{j\lambda_i}^{-1}X_{j\lambda},$$
which means that each generator $X_{i\lambda}$ such that $i\not\in\II'$ can be expressed in terms of generators whose
second indices do belong to $\II$. As desired, this renders all the generators from $\Gamma\setminus\Gamma'$ redundant,
while possibly creating some new defining relations in addition to those from $\mathfrak{R}$, resulting in a
presentation of the form $\langle \Gamma'\pre\mathfrak{R}''\rangle$ for $G_e$, where
$\mathfrak{R}''\supseteq\mathfrak{R}$.

Since $\mathfrak{R}\supseteq\mathfrak{R}'$, it follows that $G_e$ is a homomorphic image of $G'_e$, a group presented
by $\langle\Gamma'\pre\mathfrak{R}'\rangle$. By the Gray-Ru\v skuc result (Fact \ref{GR}), $G'_e\cong\mathcal{S}_k$, so
there exists a surjective group homomorphism $\xi:\mathcal{S}_k\to G_e$. However, by Lemma \ref{L4}, there is a
surjective homomorphism $\eta:G_e\to\mathcal{S}_k$. This is possible if and only if both $\xi,\eta$ are in fact
isomorphisms, thus $G_e\cong\mathcal{S}_k$, as wanted.
\end{proof}

The cases $k\in\{0,n-1,n\}$ which remain outside of the scope of Theorem \ref{main} are easy to discuss. In both the
cases $k=0$ and $k=n$, the class $D_k$ contains a single idempotent (the identity mapping $\iota_n$ and the empty set,
respectively), resulting in a singleton $\D$-class of $\ig{\PT_n}$ and a trivial maximal subgroup. Finally, for $k=n-1$
neither the class $D_{n-1}$ contains any $2\times 2$ squares (which is easy to verify), nor there are any idempotents
in $D_{n-1}\cup D_n$ available for singularization. Hence, the presentation of the corresponding maximal subgroup
consists only of relations of type (1) and (2), implying the resulting group is free.

\section*{Acknowledgements}
The author gratefully acknowledges the support of the Ministry of Education and Science of the Republic of Serbia,
through Grant No.174019, and of the Secretariat of Science and Technological Development of the Autonomous Province of
Vojvodina (Grant contract 114--451--2002/2011). Also, I am indebted to the anonymous referee for a number of useful
suggestions that improved the article.


\end{document}